\documentclass[12pt]{amsart}
\usepackage[margin=4cm]{geometry}
\usepackage{amscd}
\usepackage{amsfonts,amssymb,amsthm}

\theoremstyle{plain}
\newtheorem{thm}{Theorem}[section]
\newtheorem{lem}[thm]{Lemma}

\newtheorem{prop}[thm]{Proposition}

\theoremstyle{definition}
\newtheorem*{defn}{Definition}

\theoremstyle{remark}

\newcommand{\bbf}{\mathbb}

\newcommand{\bksl}{\setminus}

\newcommand{\sss}{{\bbf S}}
\newcommand{\aaa}{{\bbf A}}

\DeclareMathOperator{\diam}{diam}

\DeclareMathOperator{\inj}{\inj}

\DeclareMathOperator{\Out}{Out}

\begin{document}

\title[Projections of the sphere graph]{Projections of the sphere graph
to the arc graph of a surface}
\author{Brian H. Bowditch, Francesca Iezzi}
\address{Mathematics Institute, University of Warwick,
Coventry, CV4 7AL, Great Britain}
\date{26th July 2016}

\begin{abstract}
Let $ S $ be a compact surface, and $ M $ be the double of a handlebody.
Given a homotopy class of maps from $ S $ to $ M $ inducing an
isomorphism of fundamental groups, we describe a canonical uniformly
lipschitz retraction of the sphere graph of $ M $ to the arc graph
of $ S $.
We also show that this retraction is a uniformly bounded distance from
the nearest point projection map.
\end{abstract}

\maketitle

\section{Introduction}\label{SA}

Let $ H $ be a handlebody of genus $ g $, and $ M $ its double.
In other words, $ M $ is homeomorphic to a connected sum of $ g $ copies of
$ S^1 \times S^2 $.
(We refer to \cite{Hem} for general background on 3-manifolds.)
The \emph{sphere graph}, $ \sss = \sss(M) $, associated to
$ M $ can be defined as follows.
Its vertex set, $ V(\sss) $, is the set of homotopy classes of essential
2-spheres embedded in $ M $.
(A 2-sphere is ``essential'' if it does not bound a ball in $ M $.)
Two such spheres are deemed to be adjacent in $ \sss $ if they can be
homotoped to be disjoint in $ M $.
We endow $\sss$ with a path metric, $d_{\sss}$, by giving each edge length one.
It is not hard to see (by a surgery argument) that $ \sss $ is connected
\cite{Hat}.

The sphere graph has played a significant role in studying the
geometry of the outer automorphism group, $ \Out(F_g) $, of the
free group, $ F_g $, on $ g $ generators.
Note that $ \pi_1(M) \cong \pi_1(H) \cong F_g $.
Any element of $ \Out(F_g) $ can be realised by a self-homeomorphism
of $ M $ \cite{L2}, and it is not hard to see that this gives rise
to a cofinite action of $ \Out(F_g) $ on $ \sss $.
In fact, $ \sss $ is canonically isomorphic to the free-splitting graph
of $ F_g $ \cite{AS}.
A key result in the subject is that this graph is hyperbolic \cite{HanM}.
Another proof of this, directly using surgery on spheres,
has been given in \cite{HilH}.

It is also known that $ \sss $ has infinite diameter.
This is shown in \cite{HamH} by describing an isometric embedding of
the arc graph of a surface into $ \sss $.
The fact that the embedding is isometric is shown by defining a
1-lipschitz retraction to the image.
The latter construction however is not canonical.
In this paper, we show that one can define a coarsely lipschitz retraction
in a simple canonical way (see Theorem \ref{B6}).
This suffices to show that the image is quasiconvex in $ \sss $.
Another construction, by different methods, has recently been given
independently in \cite{F}.

To describe our construction, let $ S $ be any compact orientable surface,
with non-empty boundary,
$ \partial S $, and with fundamental group isomorphic to $ F_g $.
Thus, $ H \cong S \times [0,1] $.
An ``arc'', $ \alpha $, in $ S $ will be assumed to satisfy
$ \partial \alpha = \alpha \cap \partial S $, and homotopies thereof
will be assumed to slide the endpoints, $ \partial \alpha $, in $ \partial S $.
We say that $ \alpha $ is \emph{trivial} if it cuts of a disc of $ S $
(or equivalently, can be homotoped into $ \partial S $); otherwise
it is \emph{essential}.
The \emph{arc graph} $ \aaa = \aaa(S) $, is defined as follows.
Its vertex set, $ V(\aaa) $ is the set of homotopy classes of
essential arcs.
Two such arcs are deemed adjacent in $ {\aaa} $ if they can be
homotoped to be disjoint in $ S $.
The graph $\aaa$ is endowed
with a path metric, $d_{\aaa}$, by giving each edge length one.
The arc graph is known to be hyperbolic
(\cite{MS}, \cite{HenPW}, \cite {HilH}).
It also has infinite diameter (since it admits a coarsely lipschitz map
to the curve graph with cobounded image).

As in \cite{HamH}, we define a map, $ \iota : \aaa \longrightarrow \sss $,
as follows.
If $ \alpha $ is an arc in $ S $, then $ \alpha \times [0,1] $
is a disc in $ H $, which doubles to give a sphere, $ \sigma $,
in $ M $.
It is easily checked that if $ \alpha $ is essential in $ S $,
then $ \sigma $ is essential in $ M $.
Moreover, homotopic arcs give homotopic spheres, and clearly
disjoint arcs give disjoint spheres.
Therefore this gives rise to a map $ \iota $ of graphs.
In fact, $ \iota $ is also injective.
This is not hard to see directly, and will also follow from
Theorem \ref{B7} below.

We want to define a coarsely lipschitz map $ \phi : \sss \longrightarrow \aaa $
which is a left inverse (that is, $ \phi \circ \iota $ is the
identity).
The idea behind the construction is quite simple.
Note that we can embed $ S $ into $ M $ via,
$ S \cong S \times \{ 0 \} \subseteq S \times [0,1] \equiv H \subseteq M $.
This induces an isomorphism of fundamental groups.
Let $ \sigma $ be an embedded sphere in $ M $, which we take to
be in general position with respect to $ S $, so that $ \sigma \cap S $
is a collection of arcs and closed curves.
We assume that the number of arcs in $ \sigma \cap S $ is minimal
in the homotopy class of $ \sigma $.
We choose any component, $ \alpha $, of $ \sigma \cap S $ which is
an essential arc.
(It is easily seen that there must be at least one.)
We will show that, in fact, $ \alpha $ is well defined up to bounded
distance.
Also homotopically disjoint spheres can be realised to be simultaneously disjoint,
and to both intersect $ \partial S $ minimally.
In this way, we will get our desired map, $ \phi $.

To achieve this, we will shift perspective.
By a theorem of Laudenbach \cite{L1}, two essential embedded spheres in $ M $
are homotopic if and only if they are isotopic, or equivalently in this case,
ambient isotopic \cite{Hir}.
For the purposes of proving the above, one can therefore hold $ \sigma $
fixed and isotope $ S $ in $ M $.
In fact, it will be more convenient to allow singular surfaces and homotopy.
In other words, let $ f : S \longrightarrow M $ be a map inducing
an isomorphism of fundamental groups.
We assume $ f $ to be in general position with respect to $ \sigma $,
and such that $ |f^{-1}(\sigma) \cap \partial S| $ is minimal in the homotopy class
of $ f $.
Let $ \alpha $ be an essential arc component of $ f^{-1}(\sigma) \subseteq S $.
We aim to show that $ \alpha $ is well defined up to bounded distance in
$ \aaa $.
This will be based on a result in \cite{MS}, see Lemma \ref{B4} here.

We will then show:

\begin{thm}\label{A1}
For all $ \sigma, \sigma' \in V(\sss) $, we have
$ d_\aaa(\phi (\sigma), \phi (\sigma')) \le d_\sss(\sigma,\sigma') + 6 $.
\end{thm}

This implies that $ \iota \aaa $ is quasi-convex in $ \sss $.
In fact, a refinement of the argument will recover the statement
of \cite{HamH} that $ \iota $ is an isometric embedding
(see Theorem \ref{B7} here).

Note that, while $ \phi $ is shown to be well defined up to bounded
distance, it involved making a choice.
At the cost of making it multi-valued, it can be made canonical.
In particular, given $ \sigma \in V(\sss) $, we can canonically define
$ \Phi(\sigma) \subseteq \aaa $ to be the set of all arcs that
can arise from any choice of such $ f $.

We remark that the construction in \cite{HamH} is related.
However their choice of allowable maps $ f : S \longrightarrow M $ is
more restrictive.
In particular, it makes reference to a preferred arc system, and therefore
is not canonical.
Nevertheless, since it is a case of our more general construction, it follows
from the results here that it is canonical up to bounded distance.

In Section \ref{section 3}, we will relate this to nearest point projection.

Given $ \sigma \in V(\sss) $,
let $ \Pi(\sigma) = \{ \sigma' \in \iota {\aaa} \mid
d_{\sss}(\sigma,\sigma') = d_{\sss}(\sigma,\iota \aaa) \} $.
In other words, $ \Pi $ is the coarse nearest-point projection to
$ \iota \aaa $.
As with any (quasi)convex subset of a hyperbolic space, we know that
$ \Pi(\sigma) $ has bounded diameter.
Moreover, if $ \sigma,\sigma'\in V(\sss) $ are adjacent, then
$ \diam(\Pi(\sigma) \cup \Pi(\sigma')) $ is bounded,
i.e.\ $ \Pi $ is coarsely lipschitz.
Here, the constants only depend on the genus, $ g $, of $ M $.

Denote the map $ \iota \circ \Phi :\sss \longrightarrow \sss$ by $\Psi$.
In fact, we show that $ \Psi $ and $ \Pi $ agree up to bounded distance:

\begin{thm}\label{A2}
Given any $ \sigma \in V(\sss) $, $ \diam(\Psi(\sigma) \cup \Pi(\sigma)) $
is bounded above in terms of $ g $.
\end{thm}

The key ingredient for the proof is the result of \cite{HilH} that
surgery paths in $ \sss $ are quasigeodesic
(stated as Theorem \ref{C2} here).

Of course, it retrospectively follows from Theorem \ref{A2} that
$ \phi $ is coarsely lipschitz, but we know of no argument to
show that $ \iota \aaa $ is quasiconvex without first constructing
such a retraction.

\vskip0.4cm

We thank Saul Schleimer for inspiring conversations, 
in particular explaining to us his work with Masur.
We thank Arnaud Hilion for suggesting that
surgery sequences may be useful in describing nearest-point
projection maps.
We also thank Sebastian Hensel and Koji Fujiwara for related discussions.
Some of this work was carried out while both authors were visiting the
Tokyo Institute of Technology.
We are grateful to that institution for its generous support, and to Sadayoshi
Kojima for his invitation.
The second author was supported by an EPSRC Doctoral Training Award
and is currently supported by Warwick Institute of Advanced Studies
and Institute of Advanced Teaching and Learning.

\section{The main construction}\label{SB}

In this section, we fix a non-empty subset, $ \Sigma \subseteq M $,
which is a disjoint union of pairwise non-homotopic essential embedded
2-spheres.

Let $ \gamma $ be a closed multicurve (a non-empty disjoint union
of closed curves) in $ M $.
Up to small homotopy, we can assume $ \gamma $ to be embedded,
and write $ \gamma \subseteq M $.
We will assume that $ \gamma $ is in general position with respect to
$ \Sigma $, so that $ |\gamma \cap \partial \Sigma| < \infty $.

\begin{defn}
We say that $ \gamma $ is \emph{efficient} with respect to $ \Sigma $
if and only if $ |\gamma \cap \Sigma| $ is minimal in the homotopy class
of $ \gamma $.
\end{defn}

It is easily seen that $ \gamma $ is efficient if and only if each component
of $ \gamma $ is efficient.
Also, if $ \gamma $ is efficient with respect to each component of
$ \Sigma $, then it is efficient with respect to $ \Sigma $.
In fact, we have a converse:

\begin{lem}\label{B1}
If $ \gamma $ is efficient with respect to $ \Sigma $, then it
is efficient with respect to each component of $ \Sigma $.
\end{lem}

\begin{proof}
Let $\sigma$ denote any component of $\Sigma$.
If $ \gamma $ is not efficient with respect to $ \sigma $,
then there are arcs $ a \subseteq \sigma $ and $ b \subseteq \gamma $,
with the same endpoints such that $ a \cup b $ is null homotopic in $ M $.
By homotoping $ b $ to $ a $ and then off $ \sigma $,
we can reduce by at least two the number of intersections with $\sigma$.
Since the arc $b$ is disjoint from any other component of $\Sigma$,
then the homotopy does not increase the number of intersections
between the curve and the other components of $\Sigma$.
Therefore $|\gamma \cap \Sigma|$ is not minimal over the homotopy class of $\gamma$.
\end{proof}

For future reference, we will refer to a curve $ a \cup b $ as in the
above proof as an \emph{inefficient bigon}.

Now, let $ S $ be a surface which admits a map,
$ f : S \longrightarrow M $, which induces an isomorphism of fundamental groups.
For the rest of this paper, we will fix a homotopy class of such
maps, and assume that all such maps belong to this class.
We will assume $ f $ to be in general position with respect
to $\Sigma$ and $ f|\partial S$ to be an embedding.

\begin{defn}
We say that $ f : S \longrightarrow M $ is efficient with respect
to $ \Sigma$ if $ f(\partial S)$ is efficient.
\end{defn}

Note that any homotopy of $\partial S$ extends to a homotopy of $S$,
and so this is equivalent to asserting that the number of arcs in
$ f^{-1} (\Sigma) $ is minimal (in the homotopy class of $f$).
Note also that the earlier discussion carries over to efficient maps
of $ S $, i.e.\ $f$ is efficient with respect to $\Sigma$ if and only
if it is efficient with respect to each component of $\Sigma$.
Note that, $ f^{-1}(\Sigma) \cap \partial S \ne \varnothing $.
Otherwise $ f(S) $ could
be homotoped into $ M \bksl \Sigma $, and so could not carry the whole
of $ \pi_1(M) $.

\begin{lem}\label{B2}
If $f$ is efficient, then every arc of $ f^{-1}(\Sigma)$ is essential.
\end{lem}

\begin{proof}
Suppose by contradiction that $ f^{-1}(\Sigma) $ contains an inessential
arc $b$.
Then there is a subsegment $a$ of $\partial S$ so that $a \cup b$ bounds
a disc in $S$.
That is, $ a \cup b $ is an inefficient bigon, so as in
Lemma \ref{B1}, we can deduce that $f(\partial S)$ is not efficient with
respect to $\Sigma$.
\end{proof}

Since $ f $ is $ \pi_1 $-injective, each simple closed curve component
of $ f^{-1}(\Sigma) $ bounds a disc in $ S $.
One could remove such curves, by a simple surgery on $ S $.
However, since they play essentially no role in our arguments,
we will leave them alone.
(For most purposes, in particular in throughout this section, we can
effectively ignore them).

It is also worth noting that, if two maps $ f,f' : S \longrightarrow M $
induce the same map on fundamental groups (up to conjugacy), then
they are homotopic (since the higher homotopy groups of $ S $ are
all trivial).
However, if it happens that $ f|\partial S = f'|\partial S $, it
is not necessarily the case that one can take the homotopy to
fix $ \partial S $.
(One may need to push $ f(\partial S) $ around an essential sphere in $ M $
before getting back to the original curve.)
Again, this does not matter to us.

The main result is the following:

\begin{thm}\label{B3}
Suppose that $ f,f' : S \longrightarrow M $ are efficient maps
(in the same homotopy class).
Let $ \alpha \subseteq f^{-1} (\Sigma) $ and $ \alpha' \subseteq (f')^{-1} (\Sigma)$
be arcs.
Then $ d_{\aaa}(\alpha,\alpha') \le 7 $.
\end{thm}

The proof of Theorem \ref{B3} uses a result of Masur and Schleimer.
We refer to \cite{MS} (Lemma 12.20)
for the most general
statement and for a proof.
We state below the subcase we need.
Before stating the result, we recall that
a \emph{multidisc} is a disjoint union of embedded discs,
and two multicurves on a surface are said to \emph{intersect minimally}
if they realise the minimal number of intersections over their homotopy class.

\begin{lem}\label{B4}
(Masur, Schleimer)
Let $S$ be a surface with boundary and denote by $H$ the handlebody $S \times [0,1]$.
Let $\Delta$ be a properly embedded multidisc in $H$ and suppose that $\partial \Delta$ intersects $\partial (S \times \{0\})$ and
$\partial (S \times \{1\})$ minimally.
Let $\alpha$ be an arc in $\Delta \cap (S \times \{0\})$ and $\alpha'$ be an
arc in $\Delta \cap (S \times \{1\})$.
Then $ d_{\aaa}(\alpha,\alpha') \le 7 $.
\end{lem}

To relate this to \cite{MS}, note that, in the terminology of that paper,
the surfaces $ S_0 = S \times \{ 0 \} $ and $ S_1 = S \times \{ 1 \} $
are the ``horizontal'' boundary components of $ H $ viewed as a trivial
``$ I $-bundle'', and are ``large incompressible holes''
(see Definitions 12.14 and 5.2 thereof).
Therefore the hypotheses of their Lemma 12.20 are fulfilled.
(Here, of course, we are identifying $ S_0 $ and $ S_1 $
with $ S $.)

\begin{proof}[Proof of Theorem \ref{B3}]
By hypothesis $f$ and $f'$ are homotopic.
Thus there exists a map $F: S\times [0,1] \longrightarrow M$ so that $F|S_0 =f$
and $F|S_1 =f'$.
Denote $S\times [0,1]$ by $H$.
Note that, since $f$ and $f'$ induce isomorphisms of fundamental groups,
so does $F$.

We can suppose that $F$ is in general position with respect to $\Sigma$
so that $F^{-1}(\Sigma) \cap \partial H$ is an embedded multicurve $\delta$ in $\partial H$. Note that 
 $f^{-1}(\Sigma)$ coincides with $\delta \cap S_0$ and
$(f')^{-1}(\Sigma)$ coincides with $\delta \cap S_1$.

Now $ \delta $ intersects $ \partial S \times \{ 0,1 \} $ minimally.
For if not, we could find arcs $ a \subseteq \partial S \times \{ 0,1 \} $,
and $ b \subseteq \delta $, so that $ a \cup b $ bounds a disc in
$ \partial H $.
Mapping to $ M $ via $ F $ gives us an inefficient bigon, for
one of $ f(\partial S) $ or $ f'(\partial S) $ (cf.\ Lemma \ref{B1}).
But since both these maps were assumed efficient, this gives a contradiction.

Since $ f(\delta) \subseteq \Sigma $ and $F$ induces an isomorphism
of fundamental groups, each component of $\delta$ is null homotopic in $H$.
By Dehn's Lemma each component of $\delta$ bounds an embedded disc in $H$.

Now, if $\alpha$ and $\alpha'$ are as in the hypothesis, then we can apply
Lemma \ref{B4} and conclude that $d_{\aaa}(\alpha,\alpha') \le 7$.
\end{proof}

Let $ \Phi(\Sigma) \subseteq \aaa $ be the set of arcs contained in
$ f^{-1}(\Sigma) $ for some efficient map, $ f $.
(That is, we include all such arcs for all such maps, $ f $, in the
given homotopy class.).
This is non-empty, and by Lemma \ref{B4}, $ \diam \Phi(\Sigma) \le 7 $.
Write $ \Phi(\sigma) = \Phi(\{ \sigma \}) $.
Clearly, if $ \sigma \in \Sigma $,
then $ \Phi(\sigma) \subseteq \Phi(\Sigma) $.
In particular, if $ \sigma, \sigma' $ are disjoint, then
$ \diam(\Phi(\sigma) \cup \Phi(\sigma')) \le 7 $.
Note that by Laudenbach's Theorem \cite{L1}, $ \Phi(\sigma) $ depends
only on the homotopy class of $ \sigma $ (as discussed in Section \ref{SA}).
Thus, we can view $ \Phi $ as associating to an element of $ V(\sss) $,
a subset of $ V(\aaa) $ of uniformly bounded diameter.

Furthermore, note that if $ f^{-1} (\sigma) $ consists of a single arc
(and possibly some simple closed curves) for some efficient map $ f $,
then the same must be true for any other efficient map $ f' $.
Indeed, arguing as in the proof of Theorem \ref{B2}, we see
that $ \delta \subseteq \partial H $ consists of a single curve meeting
both $ S \times \{ 0 \} $ and $ S \times \{ 1 \} $ in a single arc, and possibly some inessential curves disjoint from $\partial S \times \{0, 1\}$.
Since $ \delta $ is homotopically trivial in $ H $, it is easily seen that
the two arcs in $\delta \cap (S \times \{0, 1\})$ must represent the same element of $ \aaa $.
In other words, we have $ |\Phi(\sigma)| = 1 $, in this case.

Applying this to the case where $ \sigma = \iota(\alpha) $, as
in the introduction, we immediately get:

\begin{lem}\label{B5}
If $ \alpha \in \aaa $, then $ \Phi(\iota(\alpha)) = \{ \alpha \} $.
\end{lem}

So far, everything has been canonical.
If we choose some $ \phi(\sigma) \in \Phi(\sigma) $, we get a map
$ \phi : V(\sss) \longrightarrow V(\aaa) $.
By Lemma \ref{B5}, $ \phi \circ \iota $ is the identity on $ V(\aaa) $.
By the above discussion, this extends to a 7-lipschitz retraction,
$ \phi : \sss \longrightarrow \aaa $.

In fact, we can show that this retraction is uniformly coarsely
lipschitz with multiplicative constant 1, stated as
Theorem \ref{A1} here.

To prove this, we first observe:

\begin{lem}\label{B6}
If $f:S \longrightarrow M$ is efficient with respect to a sphere $\sigma_1$,
and $\sigma_2$ is an embedded sphere disjoint from $\sigma_1$, then there
is a map $f'$ homotopic to $f$,
efficient with respect to $\sigma_2$, and coinciding with $f$ on the
preimage of $ \sigma_1 $.
\end{lem}

\begin{proof}
As in the proof of Lemma \ref{B1}, we can move $ f $ into efficient
position with respect to $ \sigma_2 $ by eliminating inefficient bigons.
But since $ f $ is already efficient with respect to $ \sigma_1 $, none
of these bigons can meet $ \sigma_1 $.
Therefore the homotopy can be carried out on the complement of
$ \sigma_1 $.
\end{proof}

\begin{proof}[Proof of Theorem \ref{B6}]
Consider two spheres $\sigma$ and $\sigma'$ in $M$, and
let $\alpha= \phi(\sigma)$ and $\alpha'= \phi(\sigma')$.
Let $ \sigma=\sigma_0, \sigma_1 ,\ldots, \sigma_n =\sigma' $
be a geodesic in $\sss$.
From this, we will construct a path
$ \alpha = \alpha_0, \alpha_1 ,\ldots, \alpha_{n-1} $ in $ \aaa $,
with $d_{\aaa}(\alpha_{n-1}, \alpha') \le 7$.

By the definition of $\phi$ there if an efficient map
$f_0 :S \rightarrow M$ so that $\alpha_0 = \alpha$
is contained in $f_0 ^{-1}(\sigma)$.
Now, by Lemma \ref{B6} there is a map $f_1$ homotopic to $f_0$,
efficient with respect to $\sigma_1$, coinciding with $f_0$ on
the preimage of $\sigma$.
The map $f_1$ yields an arc $\alpha_1$ in $f_1 ^{-1} (\sigma_1)$,
disjoint from $\alpha_0$.
We can now continue inductively, applying Lemma \ref{B6}
to each pair of spheres $(\sigma_i, \sigma_{i+1})$ in turn and
obtain our path $ \alpha_0, \alpha_1 ,\ldots, \alpha_{n-1} $ in $\aaa$.
Now note that $\alpha_{n-1}$ is by construction an arc in $f_{n-1}^{-1}
(\sigma_{n-1})$, where $f_{n-1}$ is a map efficient with respect to $\sigma_{n-1}$.
Thus, $ \alpha_{n-1}, \alpha' \in \Phi(\sigma_{n-1} \cup \sigma') $.
Therefore, by Theorem \ref{B3}, we have $d_{\aaa}(\alpha_{n-1}, \alpha') \le 7$.
It follows that $ d_{\aaa}(\alpha, \alpha') \le (n-1)+7 = n+6 $
as required.
\end{proof}

In fact, the argument also gives another proof of the result of \cite{HamH}:

\begin{thm}\label{B7}
(Hamenst\"adt, Hensel)
$\iota: \aaa \longrightarrow \sss$ is an isometric embedding,
\end{thm}

\begin{proof}
In other words, for each pair $ \alpha, \alpha' \in \aaa$ we have
$d_{\sss} (\iota(\alpha), \iota(\alpha')) = d_{\aaa}(\alpha,\alpha')$.
To see this, set $ \sigma = \iota(\alpha) $ and $ \sigma' = \iota(\alpha') $,
and construct the path $ \alpha_0 ,\dots, \alpha_n $ as in the
proof of Theorem \ref{B6},
this time continuing one more step to give us $ \alpha_n \in \Phi(\sigma') $.
Since $ \alpha_n \in \Phi(\iota(\alpha')) $,
Lemma \ref{B5} tells us that $ \alpha' = \alpha_n $, and so the statement follows.
\end{proof}

We conclude this section with some remarks.

We have observed that the homotopy class of the map, $ f $, depends only on the
induced map of fundamental groups.
Moreover, as mentioned in Section \ref{SA}, every element of
$ \Out(F_n) $ is induced by a self-homeomorphism of $ M $.
We therefore get a natural $ \Out(F_n) $-orbit of embeddedings of
$ \aaa $ into $ \sss $.
It would be interesting to understand how these embedded convex sets
fit together on a large scale.

We also remark that our construction of the retraction could be
interpreted in terms of homotopy equivalences of $ S $ to a graph.
(Note that if $ \Sigma $ is a sphere system which cuts $ M $ into holed
spheres, then the retraction of $ M $ onto the dual graph induces an
isomorphism of fundamental groups.
We can therefore postcompose a map of $ S $ into $ M $ with such
a retraction, and obtain arcs as the preimages of midpoints of edges.)
This construction ties in with approach in \cite{F}, though the arguments given there are quite different.

\section{Nearest point projection} \label{section 3}

In this section we will give a proof of Theorem \ref{A2}. 

As noted in the introduction, a key ingredient for the
proof is the result of \cite{HilH} that
surgery paths in $ \sss $ are quasigeodesic.
To formulate this, we need some definitions.

Let $ \sigma, \tau \subseteq M $ be embedded 2-spheres in general position.
We write $ C(\sigma,\tau) $ for the set of components of
$ \sigma \cap \tau $.

\begin{defn} \label{def bigons}
A \emph{3-bigon} consists of embedded discs, $ D \subseteq \sigma $ and
$ D' \subseteq \tau $, such that $ D $ and $ D' $ meet precisely in
their common boundary and $ D \cup D' $ bounds
a ball in $ M $.
It is \emph{innermost} if $ D \cap \tau = D' \cap \sigma = D \cap D' $.\\
A \emph{2-bigon} consists of a pair of distinct curves,
$ \alpha, \beta \in C(\sigma,\tau) $, together with arcs,
$ a \subseteq \sigma $ and $ b \subseteq \tau $, connecting $ \alpha $ to
$ \beta $, and with the same endpoints, and such that $ a \cup b $
is homotopically trivial in $ M $.\\
A \emph{bigon} is a 3-bigon or a 2-bigon.
\end{defn}

\begin{defn}
We say that $ \sigma, \tau $ are in \emph{normal position} if
there are no bigons.
\end{defn}

As we explain in Section \ref{SD}, this is easily seen to be
equivalent to the notion of ``normal position'' as used
in \cite{HenOP} and \cite{HilH} (generalising the notion
of Hatcher \cite{Hat}).
We can therefore apply the results of those papers.
In particular, the following is a consequence
of Hatcher's normal position:

\begin{lem}\label{C1}
$ \sigma, \tau $ admit a realisation in normal position.
\end{lem}

In fact, in view of Laudenbach's theorem \cite{L1}, this can be achieved
while holding either $ \sigma $ or $ \tau $ fixed.
Again, this will be explained in Section \ref{SD}.\\

Suppose that $ \sigma, \sigma' $ are in normal position.
An \emph{innermost disc} in $ \sigma' $ is an embedded disc,
$ D \subseteq \sigma' $, such that $ \partial D = \sigma \cap D $.
Now $ \partial D $ cuts $ \sigma $ into two discs, $ D_1 $ and $ D_2 $.
Let $ \tau_i = D \cup D_i $.
Thus, $ \tau_1, \tau_2 $ are embedded essential 2-spheres in $ M $.
Pushing $ D $ slightly off $ \sigma' $ on the side of the disc $ D_i $,
we can realise $ \tau_i $ to be in general position with respect
to $ \sigma' $.
We can further push $ D_i $ off itself on the side of $ D $ so
that $ \tau_i \cap \sigma = \varnothing $.
This implies that $ \sigma $ is adjacent to both
$ \tau_1 $ and $ \tau_2 $ in $ \sss $.
 We will say in this case that $\tau_i$ is obtained
by \emph{surgery} of $\sigma$ along $D$.
In practice it will be convenient in later constructions not
to carry out the second pushing operation.
It will be sufficient to know that $ \sigma $ and $ \tau_i $
are homotopically disjoint, without making them actually disjoint
as subsets of $ M $.
Therefore, when referring to a ``surgery'' henceforth we will assume
that we have pushed $ D $, but not $ D_i $.

\begin{defn}
Given $ \sigma, \sigma' \in V(\sss) $, we say that $ \tau \in V(\sss) $
is obtained by \emph{surgery} on $ \sigma $ \emph{in the direction of}
$ \sigma' $, if we can find realisations of $ \sigma, \sigma' $ in
normal position, and an innermost disc, $ D $, in $ \sigma' $, such
that $ \tau \in \{ \tau_1, \tau_2 \} $ in the above construction.
\end{defn}

Note that, in general, one may need to homotope $ \tau $ further so that it
is in normal position with respect to $ \sigma' $
(but see Lemma \ref{C4} below).

\begin{defn}
Given $ \sigma, \sigma' \in V(\sss) $, a \emph{surgery path}
from $ \sigma $ to $ \sigma' $ is a sequence
$ \sigma = \sigma_0, \sigma_1 ,\ldots, \sigma_n = \sigma' $ in
$ V(\sss) $ such that for all $ i < n $, $ \sigma_{i+1} $ is
obtained by surgery on $ \sigma_i $ in the direction of $ \sigma' $.
\end{defn}

Note that, if $\tau$ is obtained by surgery on $\sigma$ in the
direction of $\sigma'$, then $|C(\tau, \sigma')| < |C(\sigma, \sigma')|$.
From this, it is not hard to see that a surgery path between two
spheres $\sigma$ and $\sigma'$ always exists
(see Lemma \ref{C4} or the discussion in Section \ref{SD}).

The following is proven in \cite{HilH} (Theorem 1.2, thereof):

\begin{thm}\label{C2}
\cite{HilH}
Surgery paths are uniform unparameterised quasigeodesics.
\end{thm}

Rather than recall the formal definitions, we just note here that
(given the hyperbolicity of $ \sss $)
this implies that any surgery path from $ \sigma $ to $ \sigma' $ is
a bounded Hausdorff distance from any geodesic in $ \sss $
from $ \sigma $ to $ \sigma' $, where the bound depends only on $ g $.

In this section, we will show:

\begin{prop}\label{C3}
Let $\Psi:\sss \longrightarrow \sss$ denote the map
$\iota \circ \Phi$, and let $ \sigma \in V(\sss) $.
For any $ \sigma' \in \Psi(\sigma) $, there is some surgery path,
$ \sigma = \sigma_0 ,\ldots, \sigma_n = \sigma' $, from $ \sigma $
to $ \sigma' $ in $ \sss $ such that the diameter of
$ \bigcup_{i=0}^n \Psi(\sigma_i) $ in $ \sss $ is at most 16.
\end{prop}

To see that this implies Theorem \ref{A2}, let $ \tau \in \Pi(\sigma)$.
Given that $ \iota \aaa $ is (quasi)convex in $ \sss $, it
is easily seen, from the hyperbolicity of $ \sss $, that $ \tau $ lies
a bounded distance from any geodesic from $ \sigma $ to any point of
$ \iota \aaa $, in particular $ \sigma' $.
By Theorem \ref{C2}, it follows that there is some $ \sigma_i $ on the surgery path given by Proposition \ref{C3} so that
$ d_{\sss}(\tau,\sigma_i) $ is bounded.
Since $ \Psi $ is coarsely lipschitz, it follows that
$ \diam(\{ \tau \} \cup \Psi(\sigma_i)) $ is bounded, and so
by Proposition \ref{C3},
$ \diam(\{ \tau \} \cup \Psi(\sigma)) $ is also bounded.
Theorem \ref{A2} now follows on observing that
$ \Pi(\sigma) \ni \tau $ has bounded diameter.\\

The proof of Proposition \ref{C3} consists of several steps.\\

As a first step, Lemma \ref{C4} shows that, if we choose innermost
discs appropriately,
there will be no need to homotope the spheres obtained in surgery in
order to achieve normal position (that is, beyond pushing the innermost
disc slightly off itself after doing the surgery).

\begin{lem} \label{C4}
Let $\sigma$ and $ \sigma' $ be two essential spheres in $M$ in normal position.
Then there are at least two distinct innermost discs in $ \sigma' $,
together with spheres, obtained by surgering $\sigma$ along each of these
respective discs, which are in normal position with respect to $ \sigma' $.
\end{lem}

\begin{proof}
We first make the elementary observation that surgery can never
create any 2-bigons.
\par
Therefore, to make sure that a sphere obtained by surgering $\sigma$ in the direction of $ \sigma' $
is in normal position with respect to $ \sigma' $, we only need to
arrange that the surgery process does not create any 3-bigons.
\par

Now, define a \emph{tube} as a ball in $M$ whose boundary has
the form $D \cup A \cup D'$ where $D$ and $D'$ are innermost discs in
$ \sigma' $ and $A$ is an annulus in $\sigma$;
we also allow the degenerate case where $D=D'$ and $A$ is the boundary of $D$.
We call $A$ the \emph{annular part} of the tube.
\par
We say that a tube $T$ is contained in a tube $T'$ if the annular part of
$T$ is contained in the annular part of $T'$.
(This is equivalent to inclusion of the respective balls.)
\par
Now, let $T$ be a tube which is maximal under containment.
Write $ \partial = D \cup A \cup D' $, as above.
Let $ D_1 \subseteq \sigma $ be the disc with boundary $ \partial D $,
on the opposite side of $ A $ (or of $ D' $ in the degenerate case).
Then $ \sigma_1 = D \cup D_1 $ is a surgered sphere, which we push off
$ D $ so that it is in general position with respect to $ \sigma' $.
We write $ {\hat D} $ for the parallel copy of $ D $.
Let $ R \subseteq M $ be the ball with
$ \partial R = D \cup A_R \cup {\hat D} $,
where $ A_R \subseteq \sigma $ is an annulus.

We claim that $\sigma_1$ and $\sigma'$ do not form any 3-bigons.
The idea is simple if $\sigma_1$ and $\sigma'$ formed a 3-bigon,
then either $\sigma$ and $\sigma'$ formed a 3-bigon,
or the tube $T$ would not be maximal, leading to a contradiction. 

In fact, suppose there is a 3-bigon, $ E,E' $, where $ E $ is
a disc in $\sigma_1$ and $E'$ is a disc in $ \sigma' $.
Thus $ E \cup E' $ is the boundary of a ball, $ B \subseteq M $.
If the disc $E$ does not contain $ {\hat D} $, then $ E,E' $
would also be a 3-bigon for $\sigma$ and $ \sigma' $,
contradicting normal position of $\sigma$ and $ \sigma' $.
If $E$ does contain $ {\hat D} $, then denote $ E \setminus {\hat D} $
by $A'$.
Now, $ T \cap R = D $, $ B \cap R = {\hat D} $ and
$ T \cap B = \varnothing $.
We see that $ T' = T \cup R \cup B $ is a tube, with
$ \partial T' = D' \cup (A \cup A_R \cup A') \cup E' $.
It strictly contains $ T $, contradicting the maximality of $ T $.

Hence $\sigma_1$ is in normal position with
respect to $ \sigma' $.
Similarly, $ D' $ gives us another sphere in normal position with respect
to $ \sigma' $.

Note that, in the case of a degenerate tube,
surgering on either side of $ D = D' $ yields spheres which are in normal position with respect to $ \sigma' $.
\par
To conclude the proof, note that we always get at least two innermost discs
providing surgered spheres in normal position with respect to $ \sigma' $. In fact,
any non-degenerate tube will furnish two such discs, while if
there are only degenerate tubes, any two disjoint innermost discs in $ \sigma' $
will serve this purpose.
\end{proof}

As a next step, the aim of the following two lemmas is to show that,
if $\sigma$ is a sphere in $\sss$ and $ \sigma' $ is a sphere in
$\Psi(\sigma)$, then $\sigma$ and $ \sigma' $ can be represented
simultaneously in normal position with respect to each other, and
in efficient position with respect to a map $f:S \longrightarrow M$,
in such a way that $\sigma \cap f(S)$ and $ \sigma' \cap f(S)$ are disjoint.
This will be of fundamental importance in the proof of Proposition \ref{C3}.

\begin{lem}\label{C5}
Suppose that $ \sigma $, $ \sigma' $ are 2-spheres in $ M $,
and that $ f_0 : S \longrightarrow M $ is in general position,
and efficient with respect to $ \sigma $.
Then there is a map $f$ homotopic to $f_0$,
efficient with respect to both $ \sigma $ and $\sigma'$, and such that the
arcs of $f^{-1}(\sigma)$ are homotopic to those of $f_0^{-1}(\sigma)$.
\end{lem}

\begin{proof}

Saying that $f_0$ is not efficient with respect to $\sigma'$ is equivalent
to saying that there is an arc $b$ in $f_0(\partial S)$
together with an arc $c$ in $\sigma'$
with the same endpoints so that $b \cup c$ is null homotopic in $M$;
we will call such
an arc $b$ a \emph{returning arc}.
We can assume that the arc $c$ intersects each circle in
$\sigma \cap \sigma'$ at most once (otherwise, we could just push it
off the disc bounded by this circle).
Now, let $h:D \longrightarrow M$ be a homotopy between the
arcs $b$ and $c$, where $D$ is a disc.

We claim that any arc in  $h^{-1}(\sigma)$ has one
extremity on $h^{-1}(b)$ and the other extremity on $h^{-1}(c)$.

By efficiency of $f_0$
with respect to $\sigma$, no arc in $h^{-1}(\sigma)$ can have both
extremities on $h^{-1}(b)$.
Suppose there is an arc $b'$ in $\sigma \cap h(D)$
having both extremities on $c$.
Let $p$ and $q$ be its extremities.
Since $\sigma$ and $\sigma'$ are in normal position, the points $p$ and $q$
lie on the same circle of $\sigma \cap \sigma'$, contradicting the
assumption that $c$ intersects each circle in $\sigma \cap \sigma'$
at most once.
This proves the claim.

Now, construct a surface, $ S_1 $, homeomorphic to $ S $, by gluing
$ D $ to $ S $.
We do so by identifying the arcs $ h^{-1} (b) \subseteq \partial D $ and  $ f_0^{-1} (b) \subseteq \partial S $ via the map $ f_0^{-1} \circ h $.
This gives $ S_1 = S \cup D $, with a natural homotopy equivalence
from $ S $ to $ S_1 $.
We define a map $f_1 = f_0 \cup h : S \cup D \longrightarrow M$,
map homotopic to $g$ (via the homotopy equivalence of $ S $ and $ S_1 $).
Moreover, the arcs in $f_1^{-1}(\sigma)$ are homotopic to those
of $f_0^{-1}(\sigma)$.
By homotoping $f_0$ to $f_1$ we have reduced the number of
returning arcs.
Continuing in the same way, we can inductively eliminate all
returning arcs and find a map
$f:S \longrightarrow M$ efficient with respect to $\sigma'$,
and such that the arcs of $f^{-1} (\sigma)$ are
homotopic to those of $f_0^{-1} (\sigma)$.
\end{proof}

\begin{lem}\label{C6}
Let $\sigma$ be a sphere in $\sss(M)$.  For any sphere
$\sigma'$ in $\Psi(\sigma)$, there exist simultaneous realisations
$ \sigma $, $ \sigma' $ and $f: S \longrightarrow M$,
such that all the following conditions hold:
\par
(A1) $\sigma$ and $\sigma'$ are in normal position;
\par
(A2) $f$ is efficient with respect to $\sigma'$;
\par
(A3) $f$ is efficient with respect to $\sigma$;
\par
(A4) $f^{-1}(\sigma) \cap f^{-1}(\sigma') = \varnothing$.
\end{lem}

\begin{proof}
Let $\sigma$ be a sphere in $\sss(M)$ and let $f_0:S \longrightarrow M$
be a map efficient with respect to $\sigma$.
Let $a$ be an arc in $f_0^{-1}(\sigma)$ and let $\sigma'$ be the
sphere $\iota(a)$. Note that in this way we obtain all the spheres in $\Psi(\sigma)$.
Up to changing $\sigma'$ by homotopy,
we can suppose that $\sigma'$ and $\sigma$ are in normal position,
i.e. they satisfy (A1).
For the rest of the proof, $ \sigma $ and $ \sigma' $ will remain fixed.
\par
Now, the map $f_0$ is efficient with respect to $\sigma$,
but not necessarily with respect to $\sigma'$.
\par
By Lemma \ref{C5}, there is a map $f$ homotopic to $f_0$,
efficient with respect to $\sigma'$, and such that the arcs of
$f^{-1}(\sigma)$ are homotopic to those of $f_0^{-1}(\sigma)$.
This gives us properties (A2) and (A3).
\par
To obtain (A4), note that, since $f$ is efficient with respect
to $\sigma'$, by Lemma \ref{B5},
$f^{-1} (\sigma')$ contains only one arc, which is
homotopic to the arc $a$ (namely that originally used
to define the homotopy class of $ \sigma' $).
We will simply denote this arc by $ a $.
\par
It remains to remove any intersection points between
$f^{-1}(\sigma)$ and $f^{-1}(\sigma')$.
\par
To this end, we define a \emph{2-gon} in $ S $ as a disc whose boundary consists of
a pair of arcs in $ f^{-1}(\sigma) $ and $ f^{-1}(\sigma') $ respectively.
We refer to the intersection points of these arcs as the \emph{vertices}
of the 2-gon.
Similarly, we define a \emph{3-gon} in $ S $, to be
a disc whose boundary consists of
a pair of arcs in $ f^{-1}(\sigma) $ and $ f^{-1}(\sigma') $
(referred to as the \emph{internal sides}) together
with an arc in $ \partial S $ (referred to as the \emph{boundary side}).
The intersection point of the two internal sides is referred to as the
\emph{vertex} of the 3-gon.
\par
Now by construction, no arc of $ f^{-1}(\sigma) $ can cross $ a $;
that is they are all homotopically disjoint or equal to $ a $.
Moreover, all curves in $f^{-1}(\sigma) $ and $ f^{-1}(\sigma')$ bound
discs in $ S $.
From this, we see that all intersection points of $ f^{-1}(\sigma) $
and $ f^{-1}(\sigma') $ are vertices of 2-gons or 3-gons.
It is therefore sufficient to eliminate these.
\par
Let $ B $ be a 2-gon, with vertices $ p,q $, with
edges $ b \subseteq f^{-1}(\sigma) $ and $ b' \subseteq f^{-1}(\sigma') $.
Since $ \sigma $ and $ \sigma' $ are in normal position,
there are no 2-bigons between them.
Therefore $ f(p) $ and $ f(q) $
have to lie on the same circle $ \gamma $ in $\sigma \cap \sigma'$.
Let $ C \subseteq \sigma $ and $ C' \subseteq \sigma' $ be
the discs bounded by $ \gamma $, respectively containing
$ f(b) $ and $ f(b') $.
(Note that these arcs need not be embedded.)
Let $ G $ be the abstract 2-sphere obtained by gluing
$ C $ and $ C' $ along $ \gamma $, and let
$ \theta : G \longrightarrow M $ be the map which combines the
inclusions of $ C $ and $ C' $.
This is locally injective, and it extends to a locally injective map,
$ \theta : G \times [-1,1] \longrightarrow M $, where we identify
$ G \equiv G \times \{ 0 \} $.
We can suppose that $ \theta^{-1} (\sigma) \cap \theta^{-1} (\sigma')
\subseteq G \times \{ 0 \} $.
Note that, $ f | \partial B $ factors through a map
$ F : \partial B \longrightarrow G $, so that
$ f|\partial B = \theta \circ F $
(where $ \partial B = b \cup b' $).
We can also assume that $ f $ is transversal to $ \theta $.
In particular, we can find an annular neighbourhood,
$ A \subseteq S \bksl \partial S $, of $ \partial B $,
and an extention, $ F : A \longrightarrow G \times [-1,1] $, of $ F $ with
$ f|A = \theta \circ F|A $, and with $ F(\partial A)
\subseteq G \times \{ -1,1 \} $.
Now $ {\hat B} = B \cup A \subseteq S $ is a disc with
$ \partial {\hat B} \subseteq \partial A $, and we can suppose
that $ F(\partial {\hat B}) \subseteq G \times \{ 1 \} $.
Let $ h : {\hat B} \longrightarrow G \times \{ 1 \} $
by any continuous map with $ h|\partial {\hat B} = F|\partial {\hat B} $.
Note that $ f|\partial {\hat B} = \theta \circ h |\partial {\hat B} $.
We now modify $ f $ by replacing $ f|{\hat B} $ by
$ \theta \circ h $.
Note that we now have $ {\hat B} \cap f^{-1}(\sigma) \cap f^{-1}(\sigma')
= \varnothing $ (since $ f({\hat B}) \cap \sigma \cap \sigma' \subseteq
\theta(G \times \{ 1 \}) \cap \sigma \cap \sigma' = \varnothing $).
We have therefore reduced $ |f^{-1}(\sigma) \cap f^{-1}(\sigma')| $.
We can perturb $ f $ sligthly to ensure that it remains in general position.
This surgery does not change the induced map of fundamental
groups, $ \pi_1(S) \longrightarrow \pi_1(M) $.
The new map $ S \longrightarrow M $ is therefore homotopic to
the original.
(The homotopy between them might not fix $ \partial S $, but this does
not matter.)
\par
Continuing in this manner, we eventually remove all 2-gons.
Note that we have not touched $ f|\partial S $ in this process,
so (A2) and (A3) remain valid.
\par
Now any 3-gon can be eliminated simply by removing a small regular
neighbourhood of it in $ S $.
Given that there are now no 2-gons, any component of $ f^{-1}(\sigma) $
or of $ f^{-1}(\sigma') $ which crosses an internal edge of the 3-gon
must terminate on the boundary edge.
It follows that this operation does not change
$ |\sigma \cap f(\partial S)| $ or $ |\sigma' \cap f(\partial S)| $,
so $ f $ remains efficient.
In other words, (A2) and (A3) again remain valid.
The process does not reintroduce any 2-gons, so continuing in this way,
we eventually remove all 3-gons.
This finally achieves (A4).
\end{proof}

The next lemma is aimed at showing that, while surgering $\sigma$
towards $\sigma'$, we can keep spheres efficient with respect to a given
map $f:S \longrightarrow M$.

Recall that an embedded curve, $ \gamma \subseteq M $, is
\emph{efficient} with respect to a sphere, $ \sigma \subseteq M $,
if $ |\gamma \cap \sigma| $ is mininal in the homotopy class of $ \gamma $.

\begin{lem}\label{C7}
Let $\sigma$ be a sphere in $M$, let $\gamma$ be an embedded curve,
and let $D$ be a disc disjoint from $\gamma$ with
$\partial D \subset \sigma$.
Suppose $ \gamma $ is efficient with respect to $ \sigma $.
Then $ \gamma $ is efficient with respect to
both spheres obtained by surgering $\sigma$ along $D$.
\end{lem}

\begin{proof}
Suppose $ \gamma $ is not efficient with repect to one of the
spheres, $ \sigma_1 $, obtained by surgery.
Then there is a returning arc, $ a \subseteq \gamma $, for $ \sigma_1 $.
That is, its endpoints lie in $ \sigma_1 $, and it is homotopic, in $ M $,
to to an arc $ b \subseteq \sigma_1 $.
But now $ b $ can be homotoped off $ D $ in $ \sigma_1 $.
It follows that $ a $ is also a returning arc for $ \sigma $,
contradicting the fact that $ \gamma $ is efficient with respect to $ \sigma $.
\end{proof}

As an immediate consequence of Lemma \ref{C7},
if $\sigma$ is a sphere in $M$ efficient with respect to a map
$f :S \longrightarrow M$, and $D$ is a disc disjoint from $\partial S$ with
$\partial D \subseteq \sigma$, then the two spheres obtained by surgering
$\sigma$ along $D$ are both efficient with respect to $f :S \longrightarrow M$.

We are now ready to prove the following:

\begin{lem} \label{C8}
Let $\sigma$ be a sphere in $\sss$, then for any $\sigma'$ in $\Psi(\sigma)$,
there exists some surgery path,
$\sigma = \sigma_0, \sigma_1 ,\dots, \sigma_n =\sigma'$,
so that for each $i$ the set $\Psi(\sigma_i)$ contains a sphere
homotopically disjoint from $\sigma'$.
\end{lem}

\begin{proof}
Let $\sigma$ be a sphere in $\sss$, let $h:S \longrightarrow M$
be a map efficient with respect to $\sigma$, let $a$ be an arc
in $h^{-1} (\sigma)$ and let $\sigma'$ be the sphere $\iota(a)$.
Note that through the process we just described we can obtain any sphere in $\Psi(\sigma)$.

Let $f :S \longrightarrow M$ be a map so that
$\sigma$, $\sigma'$, and $f$ satisfy conditions (A1)--(A4) of Lemma \ref{C6}.
For the rest of the proof, $f :S \longrightarrow M$, $\sigma$ and $\sigma'$ will be fixed.

Now, $f$ is efficient with respect to $\sigma'$, and therefore,
by Lemma \ref{B5}, $f^{-1}(\sigma')$ contains exactly one arc, call it $a'$, homotopic to $a$,
and possibly some inessential closed curves.
Denote $f(a')$ by $\beta \subseteq \sigma' $.

By property (A4) in Lemma \ref{C6}, $\beta$ is disjoint from $\sigma$.
This implies that $\beta$ is contained in at most one innermost disc
of $\sigma'$.
Hence, by Lemma \ref{C4}, there exist an innermost disc $D$
in $\sigma'$ disjoint from $f(\partial S)$, and a sphere $\sigma_1$
obtained by surgering $\sigma$ along $D$,
so that $\sigma_1$ is in normal position with respect to $\sigma'$.
(Recall we have pushed the disc $ D $ slightly off itself in forming
$ \sigma_1 $.)

Now, since $D$ is disjoint from $f(\partial S)$,
by Lemma \ref{C7} the sphere $\sigma_1$
is also efficient with respect to $f$.
\par
By construction and since the disc $D$ is disjoint from $f(\partial S)$,
each arc in $f^{-1} (\sigma_1)$ is, up to homotopy,
also contained in  $f^{-1} (\sigma)$.
This means that any arc in $f^{-1} (\sigma_1)$ is at distance at most $1$
in $\aaa$ from the arc $a$.
Consequently, by efficiency of $\sigma_1$ with respect to $f$,
the projection of $\sigma_1$ contains a sphere disjoint from $\sigma'$.

Now consider the spheres $\sigma_1$ and $\sigma'$.
They are in normal position with respect to each other, the map $f: S \longrightarrow M$
is efficient with respect to both spheres, and moreover,
$f(S) \cap \sigma_1$ is disjoint from
$f(S) \cap \sigma'$ (since $f(S) \cap \sigma$ is disjoint from $f(S) \cap \sigma'$ and $\sigma_1 \cap \sigma'$
is contained in $\sigma \cap \sigma'$).

Summarising, $\sigma_1$, $\sigma'$ and
$f:S \longrightarrow M$ also satisfy conditions (A1)--(A4) of Lemma \ref{C6}.

Therefore, using the same argument as above, we can find
an innermost disc $D'$ in $\sigma'$ disjoint from $f(\partial S)$,
and a sphere $\sigma_2$, obtained by  surgering $\sigma_1$ along $D'$,
in normal position with respect to $\sigma_1$, and efficient with
respect to $f$.
By construction arcs in $f^{-1}(\sigma_2)$ are contained, up to homotopy,
in $f^{-1}(\sigma_1)$ and therefore are homotopically disjoint from $a$,
consequently the projection of $\sigma_2$ contains a sphere homotopically
disjoint from $\sigma'$.

As above one can check that $\sigma_2$,
$\sigma'$ and $f:S \longrightarrow M$ satisfy the conditions of Lemma \ref{C6}.

We proceed inductively in this way, until we obtain a sphere
$\sigma_{n-1}$ which is disjoint from $\sigma'$.
\par
The path
$\sigma, \sigma_1, \sigma_2 ,\ldots, \sigma_{n-1}, \sigma_n=\sigma'$
is a surgery path satisfying the desired property.
\end{proof}

Proposition \ref{C3} immediately follows from Lemma \ref{C8},
having noticed that $\Psi$ is defined up to distance 7.

This completes the proof of Theorem \ref{A2}.

\section{Equivalence between definitions of normal position}\label{SD}

In this section, we briefly review the notion of ``normal position''
for two 2-spheres in a 3-manifold, which was used in Section \ref{section 3}.
Here we only consider the case where the 3-manifold, $ M $,
is a doubled handlebody.
(A similar discussion would apply to a general compact 3-manifold, by
taking the connected sum decomposition into irreducible components.)

First recall that Hatcher defined a notion of ``normal form'' for
a system of disjoint spheres with respect to a maximal sphere system.
In \cite{HenOP} this was generalised to non-maximal systems.
This gives a more symmetric notion.
Their notion is used for the surgery
construction in \cite{HilH}.
Other discussions of this notion can be found in \cite{GP} and \cite{I}.
Here, for simplicity of exposition, we only consider the intersection
of two spheres (though the arguments readily adapt to two sphere systems).

Let $ \sigma, \tau \subseteq M $ be embedded 2-spheres in general position.
Lemma 7.2 of \cite{HenOP} shows that the following is equivalent to
their notion of ``normal position'', which here can serve as a definition:

\begin{defn} \label{HOP normal form}
$ \sigma, \tau $ are in \emph{normal position} if\\
(N1) each lift of $ \sigma $
to the universal cover, $ {\tilde M} $,
meets each lift of $ \tau $ in at most one curve, and\\
(N2)  no complementary component of the union of two
such lifts is a ball in $ {\tilde M} $.
\end{defn}

It is easy to see that this is equivalent to saying that there are
no bigons, hence to saying that they are in ``normal position''
as we defined it in Section \ref{section 3}.
In fact, condition (N1)
(respectively condition (N2)) is
equivalent to saying that $\sigma$ and $\tau$ do not form any 2-bigons
(respectively any 3-bigons).

As noted in Section \ref{section 3} (Lemma \ref{C1}),
a normal position for two spheres $\sigma$ and $\tau$ always exists.
This is a straightforward consequence of Proposition 1.1 of \cite{Hat}
and Lemma 7.2 in \cite{HenOP}.

Indeed, normal position is equivalent to minimal intersection of two spheres.
This can be deduced from the discussion in \cite{GP}.
However,  for the sake of completeness,
we give below another proof.

To clarify terminology, recall that
$ C(\sigma,\tau) $ is the set of components of $ \sigma \cap \tau $.
Let $ \kappa(\sigma,\tau) $ be the minimum of $ |C(\sigma,\tau)| $,
as $ \sigma $ and $ \tau $ vary over their respective homotopy classes.
Note that, in view of Laudenbach's theorem, in order to define
$ \kappa(\sigma,\tau) $, we could hold $ \sigma $ fixed,
and just allow $ \tau $ to vary in its homotopy class.

\begin{lem}\label{D1}
$ \sigma, \tau $ are in normal position if and only if
$ |C(\sigma,\tau)| = \kappa(\sigma,\tau) $.
\end{lem}

\begin{proof}
We first observe that if there are no innermost 3-bigons,
then there no 3-bigons at all.
Moreover, we can always surger out any innermost 3-bigon and reduce
$ |C(\sigma,\tau)| $.
We can therefore assume that there are never any 3-bigons.

To conclude the proof, we lift to the universal cover $ {\tilde M} $.
(This is a 3-sphere minus an unknotted Cantor set.)
Let $ s \subseteq {\tilde M} $ be any component of the lift of $ \sigma $,
which, as noted above, we can assume to be fixed.
Let $ T(\tau) $ be the set of lifts of $ \tau $, and let
$ T_0(\tau) = \{ t \in T(\tau) \mid t \cap s \ne \varnothing \} $.
Note that $ |T_0(\tau)| \le |C(\sigma,\tau)| $,
with equality if and only if there are no 2-bigons.

Suppose first, that there are no 2-bigons.
In this case, we see that if $ t \in T_0(\tau) $,
then $ t' \cap s \ne \varnothing $ for any 2-sphere, $ t' $, homotopic
to $ t $ in $ {\tilde M} $.
(This follows since the sets of ends of $ {\tilde M} $ separated
by $ s $ and $ t $ are non-nested: they pairwise intersect in
four non-empty sets.
But these sets do not change on replacing $ t $ by $ t' $.)
In particular, if $ \tau' $ is homotopic to $ \tau $ in $ M $, then
$$ |C(\sigma,\tau)| = |T_0(\tau)| \le |T_0(\tau')| \le |C(\sigma,\tau')| .$$
Therefore, $ \tau $ minimises $ |C(\sigma,\tau)| $.
In other words, $ \kappa(\sigma,\tau) = C(\sigma,\tau) $.
This shows that normal position implies minimality.

Conversely, suppose that $ \sigma,\tau' $ satisfy
$ |C(\sigma,\tau')| = \kappa(\sigma,\tau') $.
Now, by Lemma \ref{C1},
we can certainly choose $ \tau $ so
that $ \sigma,\tau $ has no bigons.
But now, the above inequality must be an equality.
In particular, $ |T_0(\tau')| = |C(\sigma,\tau')| $, so there
are no bigons in $ \sigma, \tau' $ either.
\end{proof}


\begin{thebibliography}{999999}

\bibitem[AS]{AS} J.Aramayona, J.Souto,
\emph{Automorphisms of the graph of free splittings}
:  Michigan Math. J. \textbf{60} (2011) 483--493.

\bibitem[F]{F} M.Forlini,
\emph{Splittings of free groups from arcs and curves}
: preprint (2015), posted at arXiv:1511.09446.

\bibitem[GP]{GP} S.Gadgil, S.Pandit,
\emph{Algebraic and geometric intersection numbers for free groups}
: Topology Appl. \textbf{156} (2009) 1615--1619.

\bibitem[HamH]{HamH} U.Hamenst\"adt, S.Hensel,
\emph{Spheres and projections for $Out(F_n)$}:
 J. Topol. \textbf{8} (2015) 65--92.

\bibitem[HanM]{HanM} M.Handel, L.Mosher,
\emph{The free splitting complex of a free group, {I}: hyperbolicity}
: Geom. Topol. \textbf{17} (2013) 1581--1672.

\bibitem[Hat]{Hat} A.Hatcher,
\emph{Homological stability for automorphism groups of free groups}
: Comment. Math. Helv. \textbf{70} (1995) 39--62.

\bibitem[Hem]{Hem} J.Hempel,
\emph{3-manifolds}
: Annals of Mathematics Studies No. 86, Princeton University Press (1976).

\bibitem[HenOP]{HenOP}  S.Hensel, D.Osajda, P.Przytycki,
\emph{Realisation and dismantlability}
: Geom. Topol. \textbf{18} (2014) 2079--2126.

\bibitem[HenPW]{HenPW} S.Hensel, P.Przytycki, R.Webb,
\emph{1-slim triangles and uniform hyperbolicity
for arc graphs and curve graphs}
: J. Eur. Math. Soc. \textbf{17} (2015) 755--762.

\bibitem[HilH]{HilH} A.Hilion, C.Horbez,
\emph{The hyperbolicity of the sphere complex via surgery paths}
: to appear in J. Reine Angew. Math.

\bibitem[Hir]{Hir} M.W.Hirch,
\emph{Differential topology}
: Graduate Texts in Mathematics, Springer-Verlag (1976).

\bibitem[I]{I} F.Iezzi,
\emph{Sphere systems in 3-manifolds and arc graphs}
: PhD thesis, University of Warwick (2016), available at 
http://wrap.warwick.ac.uk/78841/

\bibitem[L1]{L1} F.Laudenbach,
\emph{Sur les {$2$}-sph\`eres d'une vari\'et\'e de dimension {$3$}}
: Ann. of Math. \textbf{97} (1973) 57--81.

\bibitem[L2]{L2} F.Laudenbach,
\emph{Topologie de la dimension trois: homotopie et isotopie}
: Ast\'erisque No. 12, Soci\'et\'e Math\'ematique de France, Paris (1974).

\bibitem[MS]{MS} H.Masur, S.Schleimer
\emph{The geometry of the disk complex}
: J. Amer. Math. Soc. \textbf{26} (2013) 1--62.

\end{thebibliography}
\end{document}